\documentclass{aptpub}

\usepackage{longtable}

\usepackage{amsmath,amssymb,color,tikz,epsfig,amsfonts,latexsym,bbm,mathrsfs}
\usepackage{epsfig,amsfonts,latexsym}
\usepackage{hyperref}
\usepackage[comma, sort&compress,numbers]{natbib}

\usepackage{graphicx}
\usepackage{xcolor}

\numberwithin{equation}{section}

\usetikzlibrary{matrix}
\usepackage{graphicx}
\usepackage{latexsym,amsmath,amsfonts,amssymb,amsbsy,url,bbm}

\usepackage{bbm}

\newcommand{\sas}{\text{S}\alpha\text{S}}
\newcommand{\bbr}{\mathbb{R}}
\newcommand{\bbz}{\mathbb{Z}}
\newcommand{\mbfx}{\mathbf{X}}
\newcommand{\bbe}{\mathbb{E}}
\newcommand{\cale}{\mathcal{E}}
\newcommand{\calc}{\mathcal{C}}
\newcommand{\cald}{\mathcal{D}}
\newcommand{\mbft}{\mathbf{t}}

\newcommand{\frakm}{\mathfrak{M}}

\newcommand{\mbfe}{\mathbf{e}}
\newcommand{\mbfj}{\mathbf{j}}

\newcommand{\mbfv}{\mathbf{v}}
\newcommand{\bbo}{\mathbbm{1}}
\newcommand{\bbn}{\mathbb{N}}
\newcommand{\inv}{{-1}}

\DeclareMathOperator{\dtv}{d}
\DeclareMathOperator{\dtvx}{dx}
\DeclareMathOperator{\pprob}{\mathbf{P}}
\DeclareMathOperator{\exptn}{\mathbf{E}}

\newcommand{\pprobconv}{\stackrel{p}{\longrightarrow}}

\newcommand{\eqd}{\stackrel{d}{=}}

\numberwithin{equation}{section}

\newcommand{\beq}{\begin{equation}}
\newcommand{\eeq}{\end{equation}}
\newcommand{\alns}[1]{\begin{align*}#1\end{align*}}
\newcommand{\aln}[1]{\begin{align} #1 \end{align}}
\newcommand{\been}{\begin{enumerate}}
\newcommand{\een}{\end{enumerate}}

\theoremstyle{remark}

\newtheorem{expt}{Numerical~Experiment}

\theoremstyle{definition}

\allowdisplaybreaks

\shorttitle{Large sample test for length of memory of stable fields}
\authornames{A. Bhattacharya and P. Roy}

\begin{document}

\title{A large sample test for the length of memory of stationary symmetric stable random fields via nonsingular $\bbz^\text{\lowercase{d}}$-actions}

\authorone[Centrum Wiskunde \& Informatica, Amsterdam]{Ayan Bhattacharya}
\authortwo[Indian Statistical Institute, Bangalore]{Parthanil Roy}

\addressone{Stochastics group, Centrum Wiskunde \& Informatica, Amsterdam, North Holland, 1098XG, Netherlands.}
\addresstwo{Statistics and mathematics Unit, Indian Statistical Institute, 8th Mile, Mysore Road, RVCE Post,  Bangalore 560059, India.}

\begin{abstract}
Based on the ratio of two block maxima, we propose a large sample test for the length of memory  of a stationary symmetric $\alpha$-stable discrete parameter random field. We show that the power function converges to one as the sample-size increases to infinity under various classes of alternatives having longer memory in the sense of \cite{samorodnitsky:2004a}. Ergodic theory of nonsingular $\mathbb{Z}^d$-actions play a very important role in the design and analysis of our large sample test.\\
\end{abstract}


\keywords{Long range dependence; stationary $\sas$ random field; nonsingular group action; extreme value theory; statistical hypothesis testing.\\}


\ams{ 60G10; 	60G60; 62H15}{ 37A40}


%
%
%

\section{Introduction and Preliminaries}

A random field $\mbfx = \{X(\mbft) , \mbft \in \bbz^d\}$ is called a stationary, \textit{symmetric $\alpha$-stable} ($\sas$) random field if every finite linear combination $\sum_{i=1}^k a_i X_{\mbft_i + \mathbf{s}}$ is an $\sas$ random variable whose distribution does not depend on $\mathbf{s}$. Here we shall consider the non-Gaussian case (i.e., $0 < \alpha < 2$) unless mentioned otherwise.

Long range dependence is a very important property that has been observed in many real-life processes. By long range dependence of the the random field $\mbfx$, we mean the dependence between the observations $X(\mbft)$ which are far separated in $\mbft$. This concept was introduced in order  to study the measurements of the water flow in Nile river by famous British hydrologist Hurst (see \cite{hurst:1951} and \cite{hurst:1956}). Most of the classical definitions
of long range dependence appearing in literature are based on the
second order properties (e.g.- covariance, spectral density,
variance of partial sum, etc.) of stochastic processes. For
example, one of the most widely accepted definition of long range
dependence for a stationary Gaussian process is the following: we
say that a stationary Gaussian process has \emph{long range dependence} (also known as \emph{long memory}) if
its correlation function is not summable.  In the heavy tails case, however, this definition
becomes ambiguous because correlation function may not even exist and even if it exists, it may not have enough
information about the dependence structure of the process. For a detailed discussion on long range dependence, we refer to \cite{samorodnitsky:2007} and the references therein.

In the context of stationary $\sas$ processes ($0 < \alpha <
2$), instead of looking for a substitute for correlation function, the seminal work
\cite{samorodnitsky:2004a} suggested a new approach for long range dependence through a dichotomy in the long run behavior of the partial maxima. A partition of the underlying parameter space (formally defined later) has been suggested in the aforementioned reference which causes the dichotomy. This dichotomy has been studied for $d \ge 2$ in \cite{roy:samorodnitsky:2008}. Phase transitions in many other probabilistic features of stationary $\sas$ random fields have been connected to the same partition of the parameter space; see e.g., \cite{mikosch:samorodnitsky:2000}, \cite{resnick:samorodnitsky:2004}, \cite{roy:2010}, \cite{fasen:roy:2016}, \cite{panigrahi2017maximal}.

The fact that the law of $\mbfx$ is invariant under the group action of shift transformation on the index set $\bbz^d$ (stationarity) and certain rigidity properties of $L^\alpha$ spaces ($0<\alpha < 2$) are used in  \cite{rosinski:1995} (for $d=1$) and \cite{rosinski:2000} (for  $d \ge 2$) to show that there always exists an integral representation of the form
\aln{
X(\mbft) \eqd \int_\bbe c_\mbft(x) \Big( \frac{\dtv m \circ \phi_\mbft}{\dtv m}(x) \Big)^{1/\alpha} f \circ \phi_\mbft (x) M(\dtv x), \label{chap5_eq:rosinski_representation}
}
where $M$ is an $\sas$ random measure on a standard Borel space $(\bbe, \cale)$ with $\sigma$-finite control measure $m$, $f \in L^\alpha(\bbe,m)$ (a deterministic function), $\{\phi_\mbft\}$ is a non-singular $\bbz^d$-action on $(\bbe,m)$ (i.e., each $\phi_\mbft : \bbe \to \bbe$ is measurable and invertible, $\phi_0$ is the identity map, $\phi_{\mbft_1} \circ \phi_{\mbft_2} = \phi_{\mbft_1 + \mbft_2}$ for all $\mbft_1, \mbft_2 \in \bbz^d$ and each $m \circ \phi_\mbft^{-1}$ is an equivalent measure of $m$) and $\{c_\mbft\}$ is a measurable cocycle for the nonsingular action $\{\phi_\mbft\}$ taking values in $\{+1,-1\}$ (i.e., each  $c_\mbft : \bbe \to \{+1,-1\}$ is  measurable map such that for all $\mbft_1, \mbft_2 \in \bbz^d$, $c_{\mbft_1 + \mbft_2}(x) = c_{\mbft_2}(x) c_{\mbft_1} (\phi_{\mbft_2}(x)) $ for  all $x \in \bbe$).

As a stationary $\sas$ random field can be uniquely specified in terms of a function in $L^\alpha(\bbe, m)$, a nonsingular action and a cocycle, we consider the following parameter space for a stationary $\sas$ random field
\aln{
\Theta = \bigg\{ \Big( f, \{\phi_\mbft\}, \{c_\mbft\} \Big) : f \in L^\alpha(\bbe, m), \{\phi_\mbft\} \text{ is a nonsingular action}, ~ \{ c_\mbft\} \text{ is  a cocycle} \bigg\}. \label{chap5_eq:defn_theta}
}

\noindent Now based on the nonsingular action, we can get  a decomposition of $\bbe$ (into two subsets) which is known as Hopf decomposition as described below.  A set $W$ is called a wandering set for the nonsingular $\bbz^d$-action $\{\phi_\mbft\}$ on $(\bbe,m)$ if $\{\phi_\mbft(W) : \mbft \in \bbz^d\}$ is pairwise disjoint collection of subsets of $\bbe$. Following Proposition 1.6.1 in \cite{aaronson:1997}, we get that $\bbe$ can be decomposed into two disjoint and invariant (with respect to $\{\phi_\mbft\}$) subsets $\calc$ and $\cald$ such that for some wandering set $W \subset \bbe$, $\cald = \cup_{\mbft \in \bbz^d} \phi_\mbft(W)$ and   $\calc$ does not have any wandering set of positive measure. $\calc$ and $\cald$ are called the conservative and dissipative parts of $\{\phi_\mbft\}$, respectively. If $\bbe = \calc$, then we call the nonsingular $\bbz^d$-action $\{\phi_\mbft\}$ conservative. If  $\bbe = \cald$, then   $\{\phi_\mbft\}$ is called  dissipative. An example of a dissipative $\bbz^d$-action is the shift action: take $\bbe = \mathbb{R}^d$ (with $m$ being the Lebesgue measure) and for each $\mbft \in \mathbb{Z}^d$, define $\phi_\mbft(\mathbf{s}) = \mathbf{s}+\mbft$, $\mathbf{s} \in \mathbb{R}^d$. Section~\ref{chap5_sec:asymptotics_alternative} contains examples of conservative $\bbz^d$-actions. Roughly speaking, conservative actions tend to come back often while dissipative actions tend to move away.

 Following \cite{rosinski:1995}, \cite{rosinski:2000} and \cite{roy:samorodnitsky:2008}, and denoting the integrand  in \eqref{chap5_eq:rosinski_representation} by $f_\mbft(x)$,
\aln{
X(\mbft) \eqd \int_\calc  f_\mbft(x)  M(\dtv x)    + \int_\cald  f_\mbft(x) M(\dtv x) =: X^\calc(\mbft) + X^\cald(\mbft), \;\; \; \mbft \in \bbz^d, \label{chap5_eq:decomp_stable_random_field}
}
where $\mbfx^\calc = \{X^\calc(\mbft), \mbft \in \bbz^d\}$ and $\mbfx^\cald = \{X^\cald(\mbft), \mbft \in \bbz^d\}$ are two independent stationary $\sas$ random field generated by conservative and dissipative  nonsingular $\bbz^d$-actions, respectively. It is important to note that the stationary $\sas$ random field generated by a dissipative nonsingular $\bbz^d$-action admits mixed moving average representation (see \cite{surgailis:rosinski:marandekar:cambanis:1993} and \eqref{chap5_eq:mixed_moving_avg} below).

Based on the notion of partial block maxima, it has been established in \cite{samorodnitsky:2004a} and \cite{roy:samorodnitsky:2008} that stationary $\sas$ random fields generated by conservative actions have longer memory than those generated by a nonsingular action with a non trivial dissipative part. This has  formalized the intuition that ``conservative action keeps coming back" (i.e., same value of the random measure $M$ contributes to the observations $X(\mbft)$ which are far separated in $\mbft$) and hence induces longer memory. Let for all $n \in \bbn$,
\aln{ \label{chap5_eq:defn_Mn}
Box(n) = \{ \mbfj =(j_1, \ldots, j_d) \in \bbz^d: |j_i| \le n \text{ for } 1 \le i \le d  \}
}
be the block containing the origin with size $(2n+1)^d$ in $\bbz^d$. We define the partial block maxima for the stationary $\sas$ random field $\mbfx$ as
\alns{
M_n = \max_{\mbfj \in Box(n)} |X(\mbfj)|, ~~~~~~~~~ n\in \bbn.
}

The asymptotic behaviour of the partial block maxima $M_n$ is related to the deterministic sequence
\aln{ \label{chap5_eq:defn_bn}
B_n = \bigg( \int_\bbe \max_{\mbfj \in Box(n)} |f_\mbfj(x)|^\alpha m(\dtv x) \bigg)^{1/\alpha}.
}
 Note that by Corollary 4.4.6 of \cite{samorodnitsky:taqqu:1994},  $B_n$ is completely specified by the parameters associated to the $\sas$ random field and does not depend on the choice of the integral representation. We  shall recall the results on rate of growth of $\{B_n\}$ from \cite{roy:samorodnitsky:2008} (Proposition 4.1). It is expected that the rate of growth of $B_n$ will be slower if the underlying group action is conservative. Indeed, if $\{\phi_\mbft : \mbft \in \bbz^d\}$ is conservative, then
\aln{ \label{chap5_eq:bn_rate_consevative}
\lim_{n \to \infty} \frac{1}{(2n+1)^{d/\alpha}} B_n =0.
}
In the other case, we need the mixed moving average representation to describe the limit. A stable random field is called a mixed moving average (see \cite{surgailis:rosinski:marandekar:cambanis:1993}) if it is of the form
\aln{ \label{chap5_eq:mixed_moving_avg}
\mbfx \eqd \Bigg\{ \int_{W \times \bbz^d} f(u,\mathbf{s} -\mbft) M(\dtv u, \dtv \mathbf{s}) : \mbft \in \bbz^d \Bigg\},
}
where $f \in L^\alpha(W \times \bbz^d, \nu \otimes l)$, $l$ is the counting measure on $\bbz^d$, $\nu$ is a $\sigma$-finite measure on a standard Borel space $(W, \mathcal{W})$ and the control measure $m$ of $M$ equals $\nu \otimes l$.
It was shown in \cite{rosinski:1995}, \cite{rosinski:2000} and \cite{roy:samorodnitsky:2008} that a stationary $\sas$ random field is generated by a dissipative action if and only if it is a mixed moving average with the integral representation \eqref{chap5_eq:mixed_moving_avg}. In this case,
\aln{ \label{chap5_eq:bn_rate_mixedmoving}
\lim_{n \to \infty} \frac{1}{(2n+1)^{d / \alpha}} B_n = \Big( \int_W (g(u))^\alpha \nu(du) \Big)^{1/\alpha} \in (0, \infty),
}
where for every $u \in W$
\aln{ \label{chap5_eq:g_mixedmoving}
g(u) = \max_{\mathbf{s} \in \bbz^d} |f(u,\mathbf{s})|.
}
We shall denote the right hand side of \eqref{chap5_eq:bn_rate_mixedmoving} by $K_\mbfx$ which depends solely on $\mbfx$ and not on the integral representation.

Using the above facts, it has been established that, if the $\sas$ random field is not generated by the conservative action then
\aln{ \label{chap5_eq:maxima_dissipative}
(2n+1)^{-d/\alpha} M_n \Rightarrow C_\alpha^{1/\alpha} K_\mbfx Z_\alpha,
}
where $K_\mbfx$ is as above, $Z_\alpha$ is a standard Fr\'{e}chet($\alpha$) random variable with distribution function
\alns{
\pprob(Z_\alpha \le  z) =  \begin{cases} e^{-z^{-\alpha}} & \text{ if } z>0, \\  0 & \text{ if } z \le 0,\end{cases}
}
and
\aln{
C_\alpha = \bigg( \int_0^\infty x^{-\alpha} \sin x \dtvx \bigg) = \begin{cases} \frac{1-\alpha}{\Gamma(2-\alpha) \cos (\pi \alpha/2)} ~~~~ & \text{ if } \alpha \neq 1, \\ \frac{2}{\pi} & \text{ if } \alpha =1.
 \end{cases} \label{chap5_eq:sas_constant}
}
On the other hand, if the underlying group action is conservative then
\aln{ \label{chap5_eq:maxima_conservative}
(2n+1)^{-d/\alpha} M_n \pprobconv 0.
}
See Theorem 4.3 in \cite{roy:samorodnitsky:2008} and Theorem 4.1 in \cite{samorodnitsky:2004a}.

Note that the dichotomy between \eqref{chap5_eq:maxima_dissipative} and \eqref{chap5_eq:maxima_conservative} can be justified by the intuitive reasoning that the longer memory prevents erratic changes in $X_\mbft$ causing the maxima to grow slower. In Gaussian case, this phenomenon occurs in the form of comparison lemma; see, e.g., Corollary 4.2.3 in  \cite{leadbetter:lindgren:rootzen:1983}.


The effect of a transition from conservative to dissipative actions has been investigated for various other features of stationary $\sas$ random fields.
For example, the ruin probability of negative drifted random walk with steps from a stationary ergodic stable processes, has been studied in \cite{mikosch:samorodnitsky:2000}. It has been observed that the ruin is more likely if the  group action is conservative. The point processes associated to a stationary $\sas$ random field is analysed in \cite{resnick:samorodnitsky:2004} (for $d=1$) and \cite{roy:2010} (for $d \ge 2$). It is observed that the point process converges weakly to a Poisson cluster process if the  group action is not conservative and in the conservative case, it does not remain tight due to presence of clustering. The large deviations issues for point process convergence has been addressed in \cite{fasen:roy:2016}, where different large deviation behavior is observed depending on the ergodic theoretic properties of the underlying nonsingular actions.

Stationary $\sas$ random fields have also been studied from statistical perspective (see \cite{samorodnitsky:taqqu:1994},  \cite{karcher:spodarev:2011}, \cite{karcher:shmileva:spodarev:2013}). Different inference problems associated to the long range dependence for finite and infinite variance processes has been addressed in the literature; see for example, \cite{conti:degiovani:stoev:taqqu:2008}, \cite{stoev:taqqu:2003}, \cite{montanari:taqqu:teverovsky:1999}, \cite{giraitis:taqqu:1999}, \cite{beran:bhansali:ocker:1998}, \cite{beran:1995}, \cite{robinson:1995} and references therein. There are real-life data such as teletraffic data (\cite{cappe:moulines:pequet:petropulu:yang:2002}) which exhibits heavy-tail phenomenon and long range dependence. Motivated by all these works, the decomposition of the parameter space suggested in \cite{roy:samorodnitsky:2008} and its effect on various probabilistic aspects of $\sas$ random fields, a natural question comes in mind: \textit{is it possible to design a hypothesis testing problem which will detect the presence of long memory in the observed stationary $\sas$ random field?} In the following paragraph, we  formulate the problem.

Motivated by \cite{samorodnitsky:2004a} and \cite{roy:samorodnitsky:2008} and the other related works mentioned above,  we shall consider the following decomposition of the parameter space $\Theta$ into  $\Theta_0$ and $\Theta_1$. We define $\Theta_1$ as
\aln{
\Theta_1 = \bigg\{ \Big(f, \{\phi_\mbft\}, \{c_\mbft\} \Big) \in \Theta : \{\phi_\mbft\} \text{ is  conservative}  \bigg\} \label{chap5_eq:defn_theta1}
}
and $\Theta_0 = \Theta \setminus \Theta_1 $.  In this article, our aim is to design a large sample statistical test for  testing
\aln{
H_0: \theta \in \Theta_0 \hspace{3cm} \text{ against } \hspace{3cm} H_1 : \theta \in \Theta_1 \label{chap5_eq:testing_problem}
}
where $\theta=(f, \{\phi_\mbft\}, \{c_\mbft\})$ is the parameter associated to the observed stationary $\sas$ random field defined by \eqref{chap5_eq:rosinski_representation}.

This article is organized as follows. In Section \ref{chap5_sec:test_stat}, we shall present a large sample test (based on the ratio of two appropriately scaled block maxima) for testing $H_0$ vs. $H_1$ along with its asymptotics under both null and alternative. In particular, our test will become consistent for a reasonably broad class of alternatives. Examples of such alternatives are given in Section \ref{chap5_sec:asymptotics_alternative} followed by numerical experiments in Section~\ref{sec:simulation}. Finally, proofs of our results are discussed in Section \ref{chap5_sec:asymptotics_null}.

\section{Proposed Large Sample Test Based on Block Maxima}\label{chap5_sec:test_stat}

Let $\{\mbfe_i : 1 \le i \le d\}$ be the $d$ unit vectors in $\bbz^d$ such that the $i^{th}$ component of $\mbfe_i$ is $1$ and the other components are $0$. Fix $0< \varrho < 1$. Let
$$U_n = (2n+1)^{-d/\alpha} \max_{\mbfj \in Box(n)} |X(\mbfj)|$$
and
$$V_n = (2[n^\varrho] +1)^{-d/\alpha} \max_{\mbfj \in (2n + [n^\varrho]) \mbfe_1 + Box([n^\varrho])} |X(\mbfj)|.$$
In other words, $U_n$ is the properly scaled block maxima for $Box(n)$ containing origin as the centre and $V_n$ is the properly scaled block maxima for shifted $Box([n^\varrho])$ whose centre is sufficiently separated from $Box(n)$. To test the hypotheses \eqref{chap5_eq:testing_problem}, we define the test statistic $T_n$ as the ratio of two partial block maxima $U_n$ and $V_n$, that is
\alns{
T_n = \frac{U_n}{V_n} = \bigg(\frac{2[n^\varrho]+1}{2n +1} \bigg)^{d/\alpha} \frac{\max_{\mbfj \in Box(n)} |X(\mbfj)|}{\max_{\mbfj \in (2n + [n^\varrho]) \mbfe_1 + Box([n^\varrho])} |X(\mbfj)|}.
}
We shall derive the weak limit of the test statistic $T_n$ under the null hypothesis with  the help of following theorem.

\begin{thm} \label{chap5_thm:sas_random_field_moving_average}
Suppose that the stationary $\sas$ random field $\mathbf{X}$ is generated by a non-conservative action and hence the dissipative part $X^\cald$ admits a non-trivial moving average representation  \eqref{chap5_eq:mixed_moving_avg}, then
\aln{
(U_n, V_n) \Rightarrow (Y_1, Y_2),
}
where $Y_i$'s are independent copies of $Y$ with  distribution function
\aln{
\pprob ( Y \le y ) =\begin{cases} \exp \bigg\{ - C_\alpha K_\mathbf{X}^\alpha y^{-\alpha} \bigg\} & \text{ if } y>0, \\ 0 & \text{ if } y \le 0. \end{cases}
}
and $C_\alpha$ defined in \eqref{chap5_eq:sas_constant}.
\end{thm}

\begin{cor}\label{chap5_cor:asymptotic_distn_null_stat}
Under the assumptions of Theorem \ref{chap5_thm:sas_random_field_moving_average}, $T_n \Rightarrow T$ where $T$ has the distribution function
\aln{
F_T(t) := \pprob(T \le t) = \frac{1}{1+t^{-\alpha}}. \label{chap5_eq:distn_T}
}
\end{cor}

\begin{proof}[Proof of Corollary \ref{chap5_cor:asymptotic_distn_null_stat}]
 Using continuous mapping theorem and the fact that $Y_2 >0$ almost surely, we get that
\aln{
T_n  \Rightarrow T := \frac{Y_1}{Y_2}.
}
 The distribution of $T$ will be derived using the joint distribution of $Y_1$ and $Y_2$. It is clear that, the joint probability density function is
\alns{
h_{Y_1, Y_2}(y_1, y_2) = (C_\alpha K_\mbfx^\alpha \alpha)^2 (y_1 y_2)^{-\alpha -1} e^{- C_\alpha K_X^\alpha (y_1^{-\alpha} + y_2^{-\alpha})}, ~~~~~~~~ y_1 , y_2 >0.
}
 We follow standard substitution procedure by putting $t= y_1 y_2^\inv$ and $v=y_2$ which in turn gives us $y_1 = tv$ and $y_2=v$. It is very easy to check that the associated modulus of Jacobian of transformation is $v$ as $v > 0$. Hence we get the joint distribution of $(T, Y_2)$ as
\alns{
h_{T, Y_2}(t, y_2) = (\alpha C_\alpha K_\mbfx^\alpha)^2 t^{-\alpha -1} y_2^{-2 \alpha -1} e^{ - C_\alpha K_\mbfx^\alpha y_2^{-\alpha}(1+t^{-\alpha})}, ~~~~ t> 0,   y_2 > 0.
}
 Now to get the distribution of $T$, we have to integrate on the whole range for $y_2$. Again using standard substitution
$$z = y_2^{-\alpha} (1+ t^{-\alpha}) C_\alpha K_\mbfx^\alpha$$
we get that
\aln{
h_T(t) & = \alpha \frac{t^{-\alpha -1}}{(1+ t^{-\alpha})^2} \int_0^\infty z^{2-1} e^{-z} \dtv z \nonumber \\
& = \frac{\alpha t^{-\alpha -1}}{(1+ t^{-\alpha})^2}, ~~~~~~ t>0.
}
Hence it is easy to see that \eqref{chap5_eq:distn_T} holds for all $t > 0$.
\end{proof}

We want to compute $\tau_\beta$ such that $\pprob(T< \tau_\beta) = \beta$. An easy computation yields that,
\aln{
\tau_\beta = \bigg( \frac{\beta}{1- \beta}\bigg)^{1/\alpha}. \label{chap5_eq:lower_beta_cut_off_T}
}


\begin{remark}
Note that the distance between the two blocks is not showing up in the asymptotics of $T_n$ under the null hypothesis because the shorter memory (i.e., weaker dependence) is making the two blocks almost independent in the long run. Therefore, the asymptotic null distribution of the test-statistic becomes rather simple (ratio of two i.i.d.\ random variables as seen in Corollary~\ref{chap5_cor:asymptotic_distn_null_stat}) and the computation of the critical value \eqref{chap5_eq:lower_beta_cut_off_T} becomes very easy.
\end{remark}

\begin{remark} Even though our random field has a lot of unknown parameters (more precisely, the function $f \in L^\alpha(\mathbb{E}, m)$, the cocycle $\{c_t\}_{t \in \mathbb{Z}^d}$ and the group action $\{\phi_t\}_{t \in \mathbb{Z}^d}$), only the underlying group action plays a role in the asymptotic test procedure described in this work. Even this parameter does not need to be explicitly estimated in our method of testing. Therefore, our test is free of any estimation procedure and all our asymptotic results work well without any additional correction making this test applicable to real-life situations.
\end{remark}

The following theorem gives the asymptotics for the test statistic $T_n$ for a very broad class of alternatives.

\begin{thm} \label{chap5_thm:asymptotics_alternative}
Let $\mbfx$ be generated by a conservative $\bbz^d$-action $\{\phi_\mbft\}$. If there exists an increasing sequence of positive real numbers, $\{d_n\}$ such that
\aln{
d_n = n^{d/\alpha - \eta} L(n), \label{chap5_eq:smaller_order_magnitude}
}
 where $0 < \eta \le d/\alpha$ and $L(n)$ is a slowly varying function of $n$ and $\{ d_n^\inv M_n \}_{n \ge 1}$ and  $\{ d_n M_n^\inv \}_{n \ge 1}$  are tight sequences of random variables, then we have $$T_n \pprobconv 0.$$
\end{thm}

So we reject the null hypothesis $H_0$ against the class of alternatives considered in Theorem \ref{chap5_thm:asymptotics_alternative}, if $T_n < \tau_\beta$. This gives a large sample level-$\beta$ test for $H_0$ against $H_1$. Theorem \ref{chap5_thm:asymptotics_alternative} ensures that such a test is consistent. In the following section, we shall discuss some examples which satisfy the conditions stated in above theorem. We also derive the empirical power in a few examples based on numerical experiments.

\section{Important Classes of Alternatives} \label{chap5_sec:asymptotics_alternative}

In this section, we present a few important examples from the alternative which satisfy the hypotheses of Theorem \ref{chap5_thm:asymptotics_alternative} and hence our test becomes consistent.

\begin{example}
We consider a stationary $\sas$ random field indexed by $\bbz^2$, with the $\bbz^2$ action $\{\phi_{(i,j)}\}_{(i,j) \in \bbz^2}$ on $\bbe= \bbr$ given by
\alns{
\phi_{(i,j)}(x) = x + i + j \sqrt{2}, ~~~~ x \in \bbr
}
with $m$ as  Lebesgue measure on $\bbr$. From Example 6.3 in \cite{roy:samorodnitsky:2008}, it is clear that
\alns{
\frac{1}{n^{1/\alpha}} M_n \Rightarrow \Big( (1+\sqrt{2}) C_\alpha \Big)^{1/\alpha} Z_\alpha.
}
Hence Theorem \ref{chap5_thm:asymptotics_alternative} with $d_n = n^{1/\alpha}$ applies and we get
\alns{
T_n \pprobconv 0
}
as $n \to \infty$. So the test rejects the null hypothesis $H_0$ if $T_n < \tau_\beta$ is consistent.
\end{example}

\begin{example} \label{chap5_subsec:subgaussian}

Consider a random field which has an integral representation of the following form
\aln{
X(\mbfj) = \int_{\bbr^{\bbz^d}} g_\mbfj \dtv M, ~~~~ \mbfj \in \bbz^d
}
where $M$ is an $\sas$ random measure on $\bbr^{\bbz^d}$ whose control measure $m$ is a probability measure under which the projections $\{g_\mbfj : \mbfj \in \mathbb{Z}^d \}$  are i.i.d. random variables with finite absolute $\alpha^{th}$ moment.

First we consider the case where under $m$, $\{g_\mbfj: \mbfj \in \bbz^d\}$ are  i.i.d. positive Pareto random variables with
\alns{
m(g_\mathbf{0}> x) = \begin{cases} x^{-\gamma} & \text{ if } x \ge 1, \\ 1 & \text{ if } x< 1. \end{cases}
}
for some $\gamma > \alpha$. From Example 6.1 in \cite{roy:samorodnitsky:2008}, we get that
\alns{
B_n \sim c_{p,\gamma}^{1/\alpha} 2^{d/\gamma} n^{d/\gamma} \text{ as } n \to \infty
}
for some positive constant $c_{p, \gamma}$ and $B_n^\inv M_n$ converges weakly to Frech\'{e}t random variable. So Theorem \ref{chap5_thm:asymptotics_alternative} applies with $d_n = n^{d/\gamma}$ and  we get
\alns{
T_n \pprobconv 0
}
as $n \to \infty$. Hence the level-$\beta$ test rejects $H_0$ when $T_n < \tau_\beta$, is consistent.

Now we consider the special case where under $m$, $\{g_\mbfj : \mbfj \in \bbz^d\}$ is a sequence of i.i.d. standard normal random variables. Then $\{ X_\mbfj\}_{\mbfj \in \bbz^d}$ has the same distribution as the process $\{c_\alpha A^{1/2} G_\mbfj\}_{\mbfj \in \bbz^d}$, where $ G_\mbfj$'s are i.i.d. standard Gaussian random variables, $A$ is a positive $\alpha/2$-stable random variable independent of $\{G_\mbfj : \mbfj \in \bbz^d\}$ with Laplace transform $\exptn (e^{-tA}) = e^{-t^{\alpha/2}}$ and $c_\alpha = \sqrt{2} \Big(\exptn( |G_\mathbf{0}|^\alpha)\Big)^{1/\alpha}$; see section 3.7 in \cite{samorodnitsky:taqqu:1994}.  Then from Example 6.1 in \cite{roy:samorodnitsky:2008}, we get that
\alns{
B_n \sim \sqrt{2d \log 2n}
}
such that
\alns{
B_n^\inv  M_n \Rightarrow A^{1/2},
}
which is a positive random variable.

 So we can apply Theorem \ref{chap5_thm:asymptotics_alternative} with $d_n = \sqrt{2d \log 2n}$ and obtain
\alns{
T_n \pprobconv 0
}
as $n \to \infty$. Hence level-$\beta$ test that rejects $H_0$ if $T_n < \tau_\beta$ is consistent .

\end{example}

\begin{example} \label{exmple:effective_dim}

We shall first review the basic notions and notations from \cite{roy:samorodnitsky:2008}. Note that the group $R = \{\phi_\mbft : \mbft \in \bbz^d \}$ of invertible non-singular transformations on $(\bbe,m)$ is a finitely generated abelian group. Define the group homomorphism
\alns{
\Phi : \bbz^d \to R
}
such that $\Phi(\mbft) = \phi_\mbft$ for all $\mbft \in \bbz^d$. The kernel of this group homomorphism is $\ker(\Phi) = \{\mbft \in \bbz^d : \phi_\mbft = id_\bbe\}$ where $id_\bbe$ denotes the identity map on $\bbe$. Being a subgroup of $\bbz^d$, $\ker(\Phi)$ is a free abelian group. By first isomorphism theorem of groups, we have
\alns{
R \simeq \bbz^d / \ker(\Phi).
}
Because of structure theorem for finitely generated abelian groups (Theorem 8.5 in \cite{lang:2002}),  $R$  can be written as the direct sum of a free abelian group $\bar{F}$ (the free part) and a finite abelian group $\bar{N}$ (the torsion part). So we get
$$R = \bar{F} \oplus \bar{N}.$$
We assume that $1 \le rank(\bar{F}) = p < d$. Since $\bar{F}$ is free, there exists an injective group homomorphism
$$\Psi : \bar{F} \to \bbz^d$$
such that $\Phi \circ \Psi = id_{\bar{F}}$. Clearly $F= \Psi(\bar{F})$ is a free subgroup of $\bbz^d$ of rank $p$.

$F$ should be regarded as the effective index set and its rank $p$ becomes the effective dimension of the random field. It was shown in \cite{roy:samorodnitsky:2008} that
\alns{
\frac{1}{(2n+1)^{p/\alpha}} M_n \Rightarrow \begin{cases} C_\mbfx Z_\alpha & \text{ if } \{\phi_\mbft\}_{\mbft \in F} \text{ is not conservative,} \\
0 & \text{ if }  \{\phi_\mbft\}_{\mbft \in F} \text{ is conservative.}\end{cases}
}

In the above setup, if $1\le p < d$, and $\{\phi_\mbft\}_{\mbft \in \bbz^d}$ is not conservative, then using Theorem \ref{chap5_thm:asymptotics_alternative} with $d_n = (2n+1)^{p/\alpha}$, we get that
\alns{
T_n \pprobconv 0
}
as $n \to \infty$.
In particular, the level-$\beta$ test that rejects when $T_n < \tau_\beta$ is consistent.

\end{example}

\section{Numerical Experiments} \label{sec:simulation}

In this section, we consider some examples where the underlying group action is conservative. We shall simulate the empirical power of the proposed test of level $\beta = 10\%$ in those particular cases. It will be clear from the tables below that if we use small values of $\varrho$, then the rejection will be very frequent and hence our test will become less reliable. On the other hand, a large value of $\varrho$ results in fewer rejections and hence the power decreases for each fixed $\alpha$.  We shall also observe that the empirical power decreases as $\alpha$ increases for every fixed $\varrho$. So it seems that we need to choose a smaller  value of $\varrho$ as $\alpha$ increases. So there is an inverse relation between $\varrho$ and $\alpha$. In all the examples, however, as $n$ increases, the empirical power increases to $1$ for all values of $\varrho$ and $\alpha$ confirming the consistency of the proposed test.

\begin{expt}
Consider the set up described in the Example~\ref{chap5_subsec:subgaussian}.  For the simulation purpose, we consider the following alternative representation of the sub-Gaussian random field. Suppose that $\{G_\mbfj : \mbfj \in \bbz^2\}$ is a collection  of i.i.d. standard Gaussian random variables and $A$ is a positive $\alpha/2$-stable random variable independent of the collection  $\{G_\mbfj : \mbfj \in \bbz^2\}$ with Laplace transform $\exptn(e^{-tA}) = e^{t^{\alpha/2}}$. Let $c_\alpha = \sqrt{2} \Big( \exptn (|G_0|^\alpha) \Big)^{1/\alpha}$.  The sub-Gaussian random field has the same distribution as the collection of random variables $\{ c_\alpha A^{1/2} G_\mbfj : \mbfj \in \bbz^2\}$. It easy to simulate the i.i.d. standard Gaussian random variables  and the random variable $A$ is simulated following the method given in Page~3 of \cite{weron:weron:1995}. In the following tables, we compute the empirical power of the proposed test of level $10\%$ based on the ratio of maxima taken over two disjoint blocks.

\begin{longtable}[c]{|c|c|c|c|c|c|c|}
\cline{1-7}
$\varrho$ &  \multicolumn{3}{c|}{$\alpha = .7$} & \multicolumn{3}{c|}{$\alpha = .9$}\\
 \cline{2-7}
& $n=80$ & $n=90$ & $n=100$ & $n=80$ & $n=90$ & $n=100$ \\
\hline
\hline
0.61 &  1 & 1  & 1 & 1 & 1 &  1\\
\hline
0.62 &  1 & 1 & 1 & 1 & 1 &  1 \\
\hline
0.63 & 1  & 1 & 1 & 1 & 1 &  1 \\
\hline
0.64 & 1  & 1 & 1 & 1 & 1 & 1 \\
\hline
0.65 &  1 & 1 & 1 & 1 & 1 &  1\\
\hline
0.66 &  1 & 1 & 1 & 1  & 1 & 1 \\
\hline
0.67 &  1 & 1 & 1 & 1 & 1 & 1\\
\hline
0.68 & 1  & 1 & 1 & 0.9975 & 1 & 1\\
\hline
0.69 &  1 & 1 & 1 &  0.9975 & 1  &   1\\
\hline
0.70 &  1 & 1 & 1 &  0.9875 & 0.9975 & 1 \\
\hline
\caption{Empirical power for $\alpha =0.7$ and $0.9$}
\end{longtable}

\begin{longtable}[c]{|c|c|c|c|c|c|c|}
\cline{1-7}
$\varrho$ &  \multicolumn{3}{c|}{$\alpha = 1.1$} & \multicolumn{3}{c|}{$\alpha = 1.3$}\\
 \cline{2-7}
& $n=80$ & $n=90$ & $n=100$ & $n=80$ & $n=90$ & $n=100$ \\
\hline
\hline
0.61 &  1  & 1 & 1 & 1 & 1 &  1\\
\hline
0.62 &  1 & 1 & 1 & 1 & 1 &  1 \\
\hline
0.63 & 1  & 1 & 1 & 1 & 1 &  1 \\
\hline
0.64 & 1  & 1 & 1 & 1 & 1 & 1 \\
\hline
0.65 &  1 & 1 & 1 & 1 & 1 &  1\\
\hline
0.66 &  1 & 1 & 1 & 1  & 1 & 1 \\
\hline
0.67 &  0.9975 & 1 & 1 & 1 & 1 & 1\\
\hline
0.68 & 0.9950 & 1 & 1 & 0.9750  & 0.9875 & 0.9875\\
\hline
0.69 & 0.9850  & 1 & 0.9975 &  0.9500  & 0.9575 & 0.9875\\
\hline
0.70 &  0.9600  & 0.9975 & 0.9900 &  0.8775  & 0.9375 & 0.9775 \\
\hline
\caption{Empirical power for $\alpha =1.1$ and $1.3$.}
\end{longtable}

\begin{longtable}[c]{|c|c|c|c|c|c|c|}
\cline{1-7}
$\varrho$ &  \multicolumn{3}{c|}{$\alpha = 1.5$} & \multicolumn{3}{c|}{$\alpha = 1.7$}\\
 \cline{2-7}
& $n=80$ & $n=90$ & $n=100$ & $n=80$ & $n=90$ & $n=100$ \\
\hline
\hline
0.61 &  1 & 1 &  1 & 1 & 1 & 1\\
\hline
0.62 &  1 & 1 & 1 & 1 & 1 & 0.9975 \\
\hline
0.63 & 0.9975 & 1 & 1 & 1 & 1 &  0.995 \\
\hline
0.64 & 0.9975 & 1 & 1 & 0.9925 & 1 &  0.995 \\
\hline
0.65 &  0.9925  & 1 & 1 & 0.9875 & 0.9925 & 0.9975\\
\hline
0.66 &  0.995 &  1 & 0.995 & 0.9650  & 0.9725   & 0.9875 \\
\hline
0.67 &  0.9700 & 0.9825 & 0.9975 & 0.9125  & 0.9825 & 1\\
\hline
0.68 & 0.9325 & 0.9650 & 0.9925 & 0.8700   & 0.9300 & 0.9575\\
\hline
0.69 & 0.8825 & 0.9150  & 0.9600 & 0.7625 & 0.8925  & 0.9175\\
\hline
0.70 &  0.7625 & 0.8650  & 0.9375 & 0.6800 & 0.8125 & 0.8725 \\
\hline
\caption{Empirical power for $\alpha =1.5$ and $1.7$.}
\end{longtable}

\end{expt}

\begin{expt}

In this example, we consider a stationary $\sas$ random field $\{X(\mbft) : \mbft = (t_1, t_2, t_3) \in \bbz^3\}$ admitting the following integral representation
\aln{
X(\mbft) = \int_{\mathbb{Z}} f_{(t_1,t_2,t_3)}(x) M(\dtv x) =  \int_{\mathbb{Z}}  f(x-t_1+t_2) M(\dtv x) \label{eq:exmple_eff_dim}
}
where $M$ is an $\sas$ random measure on $\bbz$ with counting measure as control measure and $f: \bbz \to \bbr$ such that
\alns{
f(u) = \begin{cases} 1 & \text{ if } u=0 \\ 0 & \text{ otherwise.} \end{cases}
}
Note that in this case, for each $\mbft = (t_1, t_2, t_3) \in \bbz^3$, $\phi_{(t_1, t_2, t_3)}(x) = (x-t_1 + t_2)$, $x \in \bbz$. This is a special case of Example~\ref{exmple:effective_dim} and the effective dimension of the underlying group action is $1$.

It is clear that for every fixed integer $c$, the random variables $X(\mbft)$ are the same as long as $\mbft=(t_1, t_2, t_3)$ lies on the plane $t_1-t_2=c$. Also, as $c$ runs over $\bbz$, these random variables form an i.i.d.\ collection. Based on this observation, we simply simulate i.i.d.\ $\sas$ random variables (following the method sated in Page~3 of \cite{weron:weron:1995}) indexed by $\bbz$ and use them appropriately for our test. The following tables contain the simulated empirical power of the proposed test conducted at $10\%$ level of significance.

\begin{longtable}[c]{|c|c|c|c|c|c|c|}
\cline{1-7}
$\varrho$ &  \multicolumn{3}{c|}{$\alpha = .7$} & \multicolumn{3}{c|}{$\alpha = .9$}  \\
 \cline{2-7}
 & $n=1000$ & $n=1500$ & $n=2000$ & $n=1000$ & $n=1500$ & $n=2000$  \\
\hline
\hline
0.61 & 0.975 & 0.965   & 0.975  & 0.9525 & 0.9675  & 0.98125  \\ \hline
0.62 & 0.955 & 0.9625  & 0.98   & 0.945  & 0.9675  & 0.98375  \\  \hline
0.63 & 0.9475 & 0.9725 & 0.975  & 0.9475 & 0.9625  & 0.97375  \\ \hline
0.64 & 0.9525 & 0.96   & 0.9675 & 0.925  & 0.95    & 0.94875  \\ \hline
0.65 & 0.9275 & 0.96   & 0.9325 & 0.9325 & 0.95625 & 0.95     \\ \hline
0.66 & 0.9175 & 0.96   & 0.9525 & 0.9425 & 0.9525  & 0.95     \\ \hline
0.67 & 0.9225 & 0.9425 & 0.9325 & 0.9125 & 0.92375 & 0.93875  \\ \hline
0.68 & 0.91 & 0.92     & 0.93   & 0.9075 & 0.9125  & 0.9425   \\ \hline
0.69 & 0.8875 & 0.915  & 0.9275 & 0.9225 & 0.92375 & 0.92875  \\ \hline
0.70 & 0.88 & 0.9075   &  0.92  & 0.9    & 0.91625 & 0.9125   \\

\hline
\caption{Empirical power for $\alpha =0.7$ and $ 0.9$.}
\end{longtable}

\begin{longtable}[c]{|c|c|c|c|c|c|c|}
\cline{1-7}
$\varrho$ &  \multicolumn{3}{c|}{$\alpha = 1.1$} & \multicolumn{3}{c|}{$\alpha = 1.3$} \\
 \cline{2-7}
 & $n=1000$ & $n=1500$ & $n=2000$ & $n=1000$ & $n=1500$ & $n=2000$  \\
\hline
\hline
0.61 & 0.95500 & 0.97250 & 0.97250 & 0.96125 & 0.98125 & 0.9825 \\ \hline
0.62 & 0.95250 & 0.9700  & 0.97625 & 0.9525  & 0.955   & 0.97375\\  \hline
0.63 & 0.9425  & 0.97125 & 0.9675  & 0.93125 & 0.9515  & 0.96875 \\ \hline
0.64 & 0.945   & 0.9675  & 0.965   & 0.9325  & 0.95625 & 0.95625 \\ \hline
0.65 & 0.925   & 0.9525  & 0.97125 & 0.9375  & 0.95875 & 0.96375 \\ \hline
0.66 & 0.9175  & 0.94375 & 0.94625 & 0.9175  & 0.93875 & 0.94375\\ \hline
0.67 & 0.91125 & 0.9425  & 0.95375 & 0.9035  & 0.935   & 0.96625 \\ \hline
0.68 & 0.91125 & 0.94125 & 0.935   & 0.88    & 0.93125 & 0.95 \\ \hline
0.69 & 0.895   & 0.90375 & 0.95    & 0.90375 & 0.935   & 0.93375 \\ \hline
0.70 & 0.86375 & 0.8875  & 0.92625 & 0.875   & 0.88    & 0.91625 \\

\hline
\caption{Empirical power for $\alpha =1.1$ and $1.3$.}
\end{longtable}

\begin{longtable}[c]{|c|c|c|c|c|c|c|}
\cline{1-7}
$\varrho$ &  \multicolumn{3}{c|}{$\alpha = 1.5$} & \multicolumn{3}{c|}{$\alpha = 1.7$} \\
 \cline{2-7}
 & $n=1000$ & $n=1500$ & $n=2000$ & $n=1000$ & $n=1500$ & $n=2000$  \\
\hline
\hline
0.61 & 0.96125 & 0.97750 & 0.98125 & 0.9625  & 0.98    & 0.98  \\ \hline
0.62 & 0.96250 & 0.97500 & 0.97125 & 0.96625 & 0.97375 & 0.97625 \\  \hline
0.63 & 0.95    & 0.95625 & 0.97625 & 0.96625 & 0.96875 & 0.95750 \\ \hline
0.64 & 0.95625 & 0.96125 & 0.96125 & 0.94875 & 0.96250 & 0.96125 \\ \hline
0.65 & 0.93500 & 0.95250 & 0.94875 & 0.9425  & 0.95375 & 0.97125 \\ \hline
0.66 & 0.91625 & 0.93625 & 0.96    & 0.93875 & 0.9575  & 0.9575   \\ \hline
0.67 & 0.92625 & 0.94125 & 0.94125 & 0.91    & 0.93625 & 0.9475  \\ \hline
0.68 & 0.90875 & 0.93875 & 0.9375  & 0.89875 & 0.91    & 0.9475 \\ \hline
0.69 & 0.91375 & 0.92625 & 0.91625 & 0.8975  & 0.95    & 0.91 \\ \hline
0.70 & 0.89375 & 0.89125 & 0.92    & 0.89125 & 0.90375 & 0.92 \\

\hline
\caption{Empirical power for $\alpha =1.5$ and $1.7$.}
\end{longtable}
\end{expt}

\begin{expt}
Next, we consider another example of stationary $\sas$ random field admitting the integral representation \eqref{eq:exmple_eff_dim} with $f: \bbz \to \bbr$ such that
\alns{
f(u) = \begin{cases} 1 & \text{ if } u= 0, -1 \\ 0 & \text{ otherwise.} \end{cases}
}
This example is similar to the previous one with the same effective dimension $1$. In this case also, for each fixed $c \in \bbz$, the collection $\{X(\mbft): t_1 - t_2 = c\}$ consists of a single random variable. However, as $c$ runs over $\bbz$, these random variables no longer remain independent. Rather, they form a moving average process of order $1$ with $\sas$ innovations and unit coefficients. Using this observation, we simulate the random filed easily. The following tables contain the simulated empirical power of the proposed test of level $10\%$.

\begin{longtable}[c]{|c|c|c|c|c|c|c|}
\cline{1-7}
$\varrho$ &  \multicolumn{3}{c|}{$\alpha = 0.7$} & \multicolumn{3}{c|}{$\alpha = 0.9$} \\
 \cline{2-7}
 & $n=1000$ & $n=1500$ & $n=2000$ & $n=1000$ & $n=1500$ & $n=2000$  \\
\hline
\hline
0.61 & 0.9675   & 0.96875  &  0.97      & 0.9575    & 0.97        & 0.97625  \\ \hline
0.62 & 0.95875  & 0.9525   & 0.9725   & 0.955     & 0.9625   & 0.98375 \\  \hline
0.63 & 0.955   & 0.9525   & 0.98         & 0.95375 & 0.95625 & 0.975  \\ \hline
0.64 & 0.955    & 0.9525    & 0.9675    & 0.94875 & 0.96375 & 0.95875  \\ \hline
0.65 & 0.92     & 0.9475    & 0.9625   & 0.94625 & 0.94625 & 0.95625  \\ \hline
0.66 & 0.91875  & 0.93875 &  0.94875 & 0.92625 & 0.94       & 0.96125 \\ \hline
0.67 & 0.92375 & 0.9375  & 0.94375 & 0.935     &  0.94      & 0.95125  \\ \hline
0.68 & 0.9175  & 0.91875  &  0.9225   & 0.91875 & 0.91875 & 0.94 \\ \hline
0.69 & 0.9       & 0.9225    & 0.9375    & 0.9075  &  0.92625 & 0.925  \\ \hline
0.70 & 0.895   & 0.905      & 0.91625  & 0.88125 & 0.8925  &  0.915 \\

\hline
\caption{Empirical power for $\alpha =0.7$ and $0.9$.}\\
\end{longtable}

\begin{longtable}[c]{|c|c|c|c|c|c|c|}
\cline{1-7}
$\varrho$ &  \multicolumn{3}{c|}{$\alpha = 1.1$} & \multicolumn{3}{c|}{$\alpha = 1.3$} \\
 \cline{2-7}
 & $n=1000$ & $n=1500$ & $n=2000$ & $n=1000$ & $n=1500$ & $n=2000$  \\
\hline
\hline
0.61          & 0.9725   & 0.98       &     0.9775                   & 0.97       & 0.98125 & 0.98125 \\ \hline
0.62         & 0.9575   & 0.95875 &     0.97125                    & 0.955     & 0.9575  & 0.9725 \\  \hline
0.63         & 0.9575   & 0.9625   &     0.9725                   & 0.955     & 0.95375 & 0.96125\\ \hline
0.64         & 0.9375   & 0.9525   &     0.96875                     & 0.95375 & 0.9725   & 0.9625\\ \hline
0.65         & 0.9425   & 0.95       &      0.96875                & 0.9375   & 0.955     & 0.9575  \\ \hline
0.66         & 0.93625 & 0.93625&      0.95625                   & 0.9325   & 0.9425   & 0.95 \\ \hline
0.67         & 0.915      & 0.9175   &       0.93875                      & 0.9275   & 0.93625 & 0.93875\\ \hline
0.68         & 0.91625  & 0.9275  &      0.95                  & 0.915      & 0.90875 & 0.9275\\ \hline
0.69         & 0.9025   & 0.93875 &      0.93375                    & 0.88875 & 0.9075   & 0.92375\\ \hline
0.70         & 0.8825   & 0.9025   &      0.92625                      & 0.885     & 0.88375 & 0.93     \\

\hline
\caption{Empirical power for $\alpha =1.1$ and $1.3$.} \\
\end{longtable}

\begin{longtable}[c]{|c|c|c|c|c|c|c|}
\cline{1-7}
$\varrho$ &  \multicolumn{3}{c|}{$\alpha = 1.5$} & \multicolumn{3}{c|}{$\alpha = 1.7$} \\
 \cline{2-7}
 & $n=1000$ & $n=1500$ & $n=2000$ & $n=1000$ & $n=1500$ & $n=2000$  \\
\hline
\hline
0.61 & 0.97875  & 0.97        & 0.9825   & 0.96125 & 0.97875 & 0.9725 \\ \hline
0.62 & 0.9575 & 0.98       & 0.9725       & 0.95      & 0.98125 &  0.9525\\  \hline
0.63 & 0.95875 & 0.97375 & 0.96625  & 0.95625 & 0.97625 &  0.95125\\ \hline
0.64 &  0.95375 & 0.96     & 0.96375    &0.93875 & 0.9625  &  0.95875\\ \hline
0.65 & 0.9475 & 0.96375 & 0.965         & 0.95125 & 0.955   &  0.96375 \\ \hline
0.66 &  0.92125 & 0.94       & 0.96375  & 0.93625 & 0.97125  &  0.94125\\ \hline
0.67 & 0.91375 & 0.955     & 0.96         & 0.90125 & 0.94875 &  0.92875\\ \hline
0.68 & 0.9        & 0.93       & 0.94875    & 0.91875 & 0.94875  &  0.92\\ \hline
0.69 & 0.9025 & 0.92625 & 0.93          & 0.90375 & 0.9225  &  0.91625 \\ \hline
0.70 & 0.87625 & 0.905     & 0.93625  & 0.86875 & 0.88625  & 0.91625 \\

\hline
\caption{Empirical power for $\alpha =1.5$ and $1.7$.}
\end{longtable}

\end{expt}

\begin{remark} For real data, we need to choose the blocksize (i.e., $\varrho \in (0,1)$) before performing this test. Even though $\alpha$ and the best performing $\varrho$ have an inverse relationship (as explained in the beginning of this section), it is observed in the above tables that $\varrho \approx 0.65$ seem to perform well for a broad class of alternatives. Therefore, in absence of further knowledge, we prescribe $\varrho = 0.65$ to be used for our test.
\end{remark}

\section{Proofs of Theorems \ref{chap5_thm:sas_random_field_moving_average} and \ref{chap5_thm:asymptotics_alternative}} \label{chap5_sec:asymptotics_null}

\subsection{Proof of Theorem \ref{chap5_thm:sas_random_field_moving_average} }


\begin{proof}[Proof of \ref{chap5_thm:sas_random_field_moving_average}]
Without loss of generality, we shall assume that $\mbfx$ admits moving average representation. This is because under our hypothesis,  we can use the decomposition \eqref{chap5_eq:decomp_stable_random_field} with a non-trivial dissipative part and the conservative part does not contribute to the maxima after scaling. In particular, this means that
\alns{
X(\mbfj) = \int_W \int_{\bbz^d} f(u, \mathbf{v}-{\mbfj}) M(\dtv u, \dtv \mbfv), ~~~~ \mbfj \in \bbz^d,
}
where $M$ is an $\sas$ random measure on $W \times \bbz^d$ with control measure $m = \nu \otimes l$ on $\mathcal{B}(W \times \bbz^d)  $ where $l$ is counting measure on $\bbz^d$. Also $f \in L^\alpha(W \times \bbz^d, \nu \otimes l)$. Let $Box(L) = \{ \mbfj \in \bbz^d : |j_1| \le L,  \ldots, |j_d| \le L\}$ i.e., it is an $L$ neighbourhood around the origin.  Define
\aln{
X(\mbfj, L) = \int_W \int_{\bbz^d} f(u, \mbfv - \mbfj) \bbo_{W \times Box(L)} (w, \mbfv - \mbfj) M(\dtv u, \dtv \mbfv)
}
for all positive integer $L$. Define
\aln{
& M_n(L) = \max \bigg\{ |X(\mbfj,L)| : \mbfj \in Box(n) \bigg\} \\
   \text{ and } ~~~~~~ &\frakm_n(L) = \max \bigg\{ |X(\mbfj, L)|: \mbfj \in (2n + [n^\varrho]) \mbfe_1 + Box([n^\varrho]) \bigg\}.
}

Fix $L \in \bbn$. It is important to observe that, as an easy consequence of Theorem 4.3 in \cite{roy:samorodnitsky:2008}, we have
\alns{
\frac{1}{(2n+1)^{d/\alpha}} M_n(L) \Rightarrow Y_1(L),
}
where $Y_1(L)$ is a positive random variable with distribution function
\aln{
\pprob( Y_1(L) \le y) = \exp \bigg\{ - C_\alpha K_{\mathbf{X}}^\alpha (L)  y^{-\alpha}\bigg\} \label{chap5_eq:distn_y_one_L}
}
with
\alns{
K_{\mathbf{X}}^\alpha (L) = \int_W \sup_{\mbfj \in Box(L)} |f(w, \mbfj)|^\alpha \nu(\dtv w).
}
Similar facts lead to the observation that $(2[n^\varrho]+1)^{- d/ \alpha} \frakm_n$ converges weakly to a random variable with same distribution as that of $Y_1(L)$. It is important to note that for all $n \ge 2L + 1$, we have $\{X(\mbfj, L) : \mbfj \in Box(n)\}$ and  $\{ X(\mbfj, L) : \mbfj \in (2n + [n^\varrho]) \mbfe_1 + Box([n^\varrho]) \}$ are independent random vectors which follows from Theorem 3.5.3 in \cite{samorodnitsky:taqqu:1994}.  So $M_n$ and $\frakm_n$ are independent for all $n \ge 2L+1$. Combining these facts we get that
\alns{
\bigg(\frac{1}{(2n+1)^{d/\alpha}} M_n(L), \frac{1}{(2[n^\varrho]+1)^{ d /\alpha}} \frakm_n (L) \bigg) \Rightarrow (Y_1(L), Y_2(L))
}
where $Y_1(L)$ and $Y_2(L)$ are independently and identically distributed with law as specified in  \eqref{chap5_eq:distn_y_one_L}. It is easy to see that as $L \to \infty$, $K_{\mathbf{X}(L)} \to K_{\mathbf{X}}$. So we have
\alns{
(Y_1(L), Y_2(L)) \Rightarrow (Y_1, Y_2)
}
as  $L \to \infty$.

Now it only remains to show that for every fixed $\epsilon>0$,
\aln{
\lim_{L \to \infty} \limsup_{n \to \infty} \pprob \bigg(\frac{1}{(2n+1)^{d/\alpha}} |M_n - M_n(L)| + \frac{1}{(2[n^\varrho] +1)^{\varrho d / \alpha}} |\frakm_n - \frakm_n(L)| > \epsilon \bigg) =0. \label{chap5_eq:conv_togather_1}
}
To show \eqref{chap5_eq:conv_togather_1}, it is enough to show that
\alns{
\lim_{L \to \infty} \limsup_{n \to \infty} \pprob \bigg( \frac{1}{(2n+1)^{d/\alpha}}  |M_n - M_n(L)| > \epsilon/2\bigg) =0.
}

Recall that
\alns{
B_n =\bigg( \int_\bbe \max_{\mbfj \in Box(n)} \Big|f(w,\mbfj) \Big|^\alpha m(\dtv w) \bigg)^{1/\alpha}
}
and define a new probability measure $\lambda_n$ on $\bbe= W \times \bbz^d$ for every fixed $n$,
\aln{
\frac{\dtv \lambda_n}{\dtv m}(w,\mbfj) = B_n^{-\alpha} \max_{\mbfj \in Box(n)} \Big|f(w,\mbfj) \Big|^\alpha.
}

Using Theorem 3.5.6 and Corollary 3.10.4  from \cite{samorodnitsky:taqqu:1994}, we know that for $\mbfj \in Box(n)$,
\alns{
X(\mbfj) \eqd C_\alpha^{1/ \alpha} \sum_{i=1}^\infty \varepsilon_i \Gamma^{-1/\alpha}_i f(U_i^{(n)}, \mathbf{V}_i^{(n)} - \mbfj), ~~~~~~~~ \mbfj \in Box(n),
}
where $C_\alpha$ is a constant as specified in \eqref{chap5_eq:sas_constant}, $\{ \varepsilon_i : i \ge 1\}$ is a collection  of i.i.d. $\{\pm 1 \}$-valued symmetric random variables, $\{\Gamma_i : i \ge 1\}$ is the collection of arrival times of the unit rate Poisson process and $\{(U_i^{(n)}, \mathbf{V}_i^{(n)} ) : i \ge 1 \}$ is a collection of i.i.d. $\bbe = W \times \bbz^d$-valued random variables with common law $\lambda_n$ for every fixed $n$. It is straight forward to check that
\alns{
X(\mbfj, L) \eqd C_\alpha^{1/ \alpha} \sum_{i=1}^\infty \varepsilon_i \Gamma^{-1/\alpha}_i f(U_i^{(n)}, \mathbf{V}_i^{(n)} - \mbfj) \bbo_{W \times Box(L)}(U_i^{(n)}, \mathbf{V}_i^{(n)} - \mbfj), ~~~~ \mbfj \in Box(n).
}
 Now note that
\aln{
& \max_{\mbfj \in Box(n)} |X(\mbfj)| - \max_{\mbfj \in Box(n)} |X(\mbfj, L)| \nonumber \\
& = \max_{\mbfj \in Box(n)} \Big|X(\mbfj,L)  + C_\alpha^{1/\alpha} \sum_{i=1}^\infty \varepsilon_i \Gamma_i^{-1/\alpha}  f(U_i^{(n)}, \mathbf{V}_i^{(n)} - \mbfj)  \bbo_{W \times (Box(L))^c} (U_i ^{(n)}, \mathbf{V}_i^{(n)} -\mbfj) \Big| \nonumber \\
 & \hspace{3cm} - \max_{\mbfj \in Box(n)} |X(\mbfj, L)| \nonumber \\
& \le \max_{\mbfj \in Box(n)} \Big|C_\alpha^{1/\alpha} \sum_{i=1}^\infty \varepsilon_i \Gamma_i^{-1/\alpha} f(U_i^{(n)}, \mathbf{V}_i^{(n)} - \mbfj) \bbo_{W \times (Box(L))^c} (U_i^{(n)} , \mathbf{V}_i^{(n)} -\mbfj) \Big| \label{chap5_eq:difference_upper}
}
where the last inequality follows from the fact that
\alns{
\max_{\mbfj \in \bbz^d}( a_\mbfj + b_\mbfj) \le \max_{\mbfj \in \bbz^d} a_{\mbfj} + \max_{\mbfj \in \bbz^d} b_{\mbfj}
}
for two sequences $\{a_\mbfj : \mbfj \in \bbz^d\}$ and $\{b_\mbfj : \mbfj \in \bbz^d \}$ of positive real numbers. Also note that
\aln{
& \max_{\mbfj \in Box(n)} |X(\mbfj)| - \max_{\mbfj \in Box(n)} |X(\mbfj, L)| \nonumber \\
& = \max_{\mbfj \in Box(n)} |X(\mbfj)|-  \max_{\mbfj \in Box(n)} \Big| X(\mbfj)   \nonumber \\
& \hspace{3.5cm} - C_\alpha^{1/\alpha} \sum_{i=1}^\infty \varepsilon_i \Gamma_i^{-1/\alpha} f(U_i^{(n)}, \mathbf{V}_i^{(n)} - \mbfj) \bbo_{W \times (Box(L))^c} (U_i^{(n)} , \mathbf{V}_i^{(n)} -\mbfj) \Big| \nonumber \\
& \ge  - \max_{\mbfj \in Box(n)} \Big|C_\alpha^{1/\alpha} \sum_{i=1}^\infty \varepsilon_i \Gamma_i^{-1/\alpha} f(U_i^{(n)}, \mathbf{V}_i^{(n)} - \mbfj) \bbo_{W \times (Box(L))^c} (U_i^{(n)} , \mathbf{V}_i^{(n)} -\mbfj) \Big| \label{chap5_eq:difference_lower}
}
using the fact that any two sequence of real numbers $\{ a_\mbfj : \mbfj \in \bbz^d \}$ and $\{b_\mbfj : \mbfj \in \bbz^d\}$ satisfy the following inequality
\alns{
\max_{\mbfj \in \bbz^d} |a_\mbfj| - \max_{\mbfj \in \bbz^d} |a_\mbfj - b_\mbfj| \ge - \max_{\mbfj \in \bbz^d} |b_\mbfj|.
}

Now combining the the upper bound in \eqref{chap5_eq:difference_upper} and the lower bound obtained in \eqref{chap5_eq:difference_lower}, we get that
\alns{
\Big|\max_{\mbfj \in Box(n)}|X(\mbfj)|  - \max_{\mbfj \in Box(n)} |X(\mbfj, L )| \Big| \le \max_{\mbfj \in Box(n)} |X^{(c)}(\mbfj, L)|,
}
where
$$X^{(c)}(\mbfj, L) =C_\alpha^{1/\alpha} \sum_{i=1}^\infty \varepsilon_i \Gamma_i^{-1/\alpha} f(U_i^{(n)}, \mathbf{V}_i^{(n)} - \mbfj) \bbo_{W \times (Box(L))^c} (U_i^{(n)} , \mathbf{V}_i^{(n)} -\mbfj). $$
It is easy to verify that $\{X^{(c)}(\mbfj, L) : \mbfj \in Box(n)\}$ is a stationary $\sas$ random field which admits mixed moving average representation. Hence we can again use Theorem 4.3  in \cite{roy:samorodnitsky:2008}, to get that
\alns{
\frac{1}{(2n+1)^{d/\alpha}}\max_{\mbfj \in Box(n)} | X^{(c)}(\mbfj, L)  | \Rightarrow C_\alpha^{1/\alpha} K^{(c)}_{\mathbf{X}}(L) Z_\alpha,
}
where $Z_\alpha$ is a Frechet random variable with distribution function
\alns{
\pprob(Z_\alpha < x) = e^{-x^{-\alpha}}
}
and
\alns{
K^{(c)}_\mathbf{X}(L) = \int_W \sup_{\mbfj \in \bbz^d \setminus Box(L)} |f(w, \mbfj)|^\alpha \nu(\dtv w).
}

Finally we have that
\aln{
& \limsup_{n \to \infty} \pprob \Big( (2n+1)^{-d/\alpha} |M_n - M_n(L)| > \epsilon/2 \Big) \nonumber \\
& = \limsup_{n \to \infty} \pprob \bigg((2n+1)^{-d/\alpha} \Big|\max_{\mbfj \in Box(n)} |X(\mbfj)| - \max_{\mbfj \in Box(n)} |X(\mbfj, L)| \Big| > \epsilon/2 \bigg) \nonumber \\
& \le \limsup_{n \to \infty} \pprob \bigg( (2n+1)^{-d/\alpha} \max_{\mbfj \in Box(n)} |X^{(c)}(\mbfj, L)| > \epsilon/2 \bigg) \nonumber \\
& = \pprob \bigg( C_\alpha^{1/\alpha} K_{\mathbf{X}}^{(c)}(L) Z_\alpha > \epsilon/2 \bigg). \label{chap5_eq:conv_togather_upper_bound}
}
It is easy to see that as $L \to \infty$, $K^{(c)}_\mathbf{X}(L) \to 0$ and hence the expression in \eqref{chap5_eq:conv_togather_upper_bound} vanishes. This completes the proof of \eqref{chap5_eq:conv_togather_1}.
\end{proof}

\subsection{Proof of Theorem \ref{chap5_thm:asymptotics_alternative}}

From the fact that $\{d_n^\inv M_n\}$ and $\{d_n M_n^\inv\}$ are tight sequences,  it follows using stationarity that $\{ \Big(d([n^\varrho])\Big)^\inv \frakm_n\}$ and $\{d([n^\varrho]) \frakm_n^\inv \}$ are tight sequences, where $d(n):=d_n$.
Note that as a product of two tight sequences,
\aln{
\frac{d([n^\varrho])}{d_n} \frac{M_n}{\frakm_n} \label{chap5_eq:prod_tight}
}
is also a tight sequence of random variables. Observe that
\alns{
\frac{d([n^\varrho])}{d(n)} \sim n^{(d/\alpha - \eta)(\varrho -1)} \frac{L([n^\rho])}{L(n)}
}
as $n \to \infty$, Note also that
\alns{
T_n & = \bigg(\frac{2[n^\varrho] + 1}{2n +1}\bigg)^{d/\alpha} \frac{M_n}{\frakm_n} \\
& \sim n^{d/\alpha (\varrho - 1)} \frac{M_n}{\frakm_n} \\
&\sim\frac{L(n)}{L([n^\varrho])} n^{\eta(\varrho-1)} \frac{d([n^\varrho])}{d(n)} \frac{M_n}{\frakm_n},
}
from which the result follows because \eqref{chap5_eq:prod_tight} is tight and
$$\frac{L(n)}{L([n^\varrho])} n^{\eta(\varrho-1)}\to 0 $$
using Potter bounds (see, e.g. \cite{resnick:1987}).

\vspace{2cm}

\noindent{\bf{Acknowledgements.}} The authors would like to thank the anonymous referees for careful reading and detailed comments which significantly improved the paper. This research was partially supported by the project RARE-318984 (a Marie Curie FP7 IRSES
Fellowship) at Indian Statistical Institute.  Parthanil Roy was also supported by Cumulative Professional Development Allowance.

\bibliographystyle{apt}

\end{document}